\documentclass[1p]{elsarticle}
\newtheorem{theorem}{Theorem}
\newtheorem{proposition}[theorem]{Proposition}
\newtheorem{corollary}[theorem]{Corollary}
\newtheorem{example}[theorem]{Example}
\newtheorem{lemma}[theorem]{Lemma}
\newtheorem{remark}[theorem]{Remark}
\newenvironment{proof}[1][Proof]{\noindent\textbf{#1.} }{\ \rule{0.5em}{0.5em}}
\newcommand{\Vertic}{\textup{Vert}}
\usepackage{makeidx}  
\usepackage{amssymb}
\usepackage{tikz}
\usetikzlibrary{arrows,shapes, shapes.geometric, patterns}
\usetikzlibrary{external}
\tikzset{external/force remake}
\RequirePackage{amsmath, amsfonts,amssymb}

\begin{document}
\begin{frontmatter}          
%
%
%
\title{On teaching sets of $k$-threshold functions}
%
%
\author[zam]{Elena Zamaraeva}
%
%
%
{\ead{elena.zamaraeva@gmail.com}}
\address{Nizhny Novgorod State University, Gagarin ave. 23, Nizhny Novgorod 603600, Russia}


%

\begin{abstract}
Let $f$ be a $\{0,1\}$-valued function over an integer $d$-dimensional cube $\{0,1,\dots,n-1\}^d$, for $n \geq 2$ and $d \geq 1$. 
The function $f$ is called threshold if there exists a hyperplane which separates $0$-valued points from $1$-valued points. Let $C$ be a class of functions and $f \in C$. 
A point $x$ is essential for the function $f$ with respect to $C$ if there exists a function $g \in C$ such that $x$ is a unique point on which $f$ differs from $g$. 
A set of points $X$ is called teaching for the function $f$ with respect to $C$ if no function in $C \setminus \{f\}$ agrees with $f$ on $X$. 
It is known that any threshold function has a unique minimal teaching set, which coincides with the set of its essential points.
In this paper we study teaching sets of $k$-threshold functions, 
i.e. functions that can be represented as a conjunction of $k$ threshold functions. 
We reveal a connection between essential points of $k$ threshold functions and essential
points of the corresponding $k$-threshold function. We note that, in general, 
a $k$-threshold function is not specified by its essential points 
and can have more than one minimal teaching set.
We show that for $d=2$ the number of minimal teaching sets for a 2-threshold function can grow as $\Omega(n^2)$. 
We also consider the class of polytopes with vertices in the $d$-dimensional cube. Each
polytope from this class can be defined by a $k$-threshold function for some $k$. 
In terms of $k$-threshold functions we prove that a polytope with vertices in the $d$-dimensional cube has a unique minimal teaching set which is equal to the set of its essential points. 
For $d=2$ we describe structure of the minimal teaching set of a polytope and show that cardinality
of this set is either $\Theta(n^2)$ or $O(n)$ and depends on the perimeter and the minimum angle of the polytope.
\end{abstract}
\end{frontmatter}

\section{Introduction}

Let $n$ and $d$ be integers such that $n \geq 2$ and $d \geq 1$ and let $E_n^d$ denote a $d$-dimensional cube $\{0, 1, \dots, n-1\}^d$.
A function $f$ that maps $E_n^d$ to $\{0,1\}$ is \emph{threshold}, 
if there exist real numbers $a_{0}, a_{1}, \ldots, a_{d}$ such that 
$$
		M_1(f)=
				\left\{ x \in E_n^d:\sum\limits_{j=1}^{d}{a_{j} x_{j}} \leq a_{0}\right\},
$$
where $M_\nu(f)$ is the set of points $x \in E_n^d$ for which $f(x)=\nu$.
The inequality
$\sum\limits_{j=1}^{d}{a_{j} x_{j}} \leq a_{0}$
is called \emph{threshold}. 
We denote by $\mathfrak{T}(d, n)$ the class of all threshold functions over $E_n^d$.

Let $k$ be a natural number. A function $f$ that maps $E_n^d$ to $\{0, 1\}$ is called \emph{$k$-threshold}
if there exist real numbers $a_{10}, a_{11}, \ldots, a_{kd}$ such that
\begin{equation}\label{eq:khalfspace1}
		M_1(f)=
				\left\{ x \in E_n^d:\sum\limits_{j=1}^{d}{a_{ij} x_{j}} \leq a_{i0}, \text{ для } i=1,\ldots,k \right\}.
\end{equation}
The system of inequalities 
$\sum\limits_{j=1}^{d}{a_{ij} x_{j}} \leq a_{i0}, \text{ для } i=1,\ldots,k$
is called \emph{threshold} and \emph{defines} the $k$-threshold function $f$.
Let $\mathfrak{T}(d, n, k)$ be the class of $k$-threshold functions over $E_n^d$. 
By definition $\mathfrak{T}(d, n, 1) = \mathfrak{T}(d, n)$. 
Note that a $k$-threshold function is also a $j$-threshold function for $j > k$. 
Denote by $\mathfrak{T}(d,n,*)$ the class of all $k$-threshold functions over $E_n^d$ for all natural $k$, that is $\mathfrak{T}(d,n,*) = \bigcup\limits_{k \geq 1}\mathfrak{T}(d,n,k)$.

For any $k$-threshold function $f$ there exist threshold functions $f_1, \dots, f_k$ such that
$$
	f(x) =f_1(x) \land \dots \land f_k(x),
$$
where "$\land$" denotes the usual logical conjunction.
We will say that $f$ is \emph{defined} by $f_1, \dots, f_k$ and $\{f_1, \dots, f_k\}$ is \emph{defining set} for $f$.


A convex hull of a set of points $X \subseteq \mathbb{R}^d$ is denoted by $\textup{Conv}(X)$.
For a function $f: E_n^d \rightarrow \{0,1\}$ we denote by $P(f)$ the convex hull of $M_1(f)$, that is $P(f) = \textup{Conv}(M_1(f))$. 
For any polytope $P$ with vertices in $E_n^d$ there exists a unique $k$-threshold function $f$, such that $P = P(f)$.
Therefore there is one-to-one correspondence between functions in the class $\mathfrak{T}(d,n,*)$
 and polytopes with vertices in $E_n^d$, and we can say that $\mathfrak{T}(d,n,*)$ is a \emph{class of polytopes with vertices in} $E_n^d$.

In \cite{Angluin} Angluin considered a model of concept learning with membership queries.
In this model a \emph{domain} $X$ and a \emph{concept class} 
$\mathcal{S} \subseteq 2 ^ X$ are known to both the \emph{learner} (or \emph{learning algorithm}) and the \emph{oracle}. 
The goal of the learner is to identify an unknown \emph{target concept} 
$S_T \in \mathcal{S}$ that has been fixed by the oracle. To this end, the learner may ask 
the oracle membership queries ``does an element $x$ belong to $S_T$?'', to which 
the oracle returns ``yes'' or ``no''. 
The learning complexity of a learning algorithm $\mathcal{A}$ with respect to a concept class $\mathcal{S}$ 
is the minimum number of membership queries sufficient for $\mathcal{A}$ to identify any concept in $\mathcal{S}$. 
The \emph{learning complexity of a concept class} $\mathcal{S}$ is defined as the minimum learning complexity of a
learning algorithm with respect to $\mathcal{S}$ over all learning algorithms which learn $\mathcal{S}$ using membership queries.


In terms of Angluin's model, a $\{0,1\}$-valued functions over $E_n^d$ can be considered as a characteristic functions of concepts. Here $E_n^d$ is the domain and a function $f: E_n^d \rightarrow \{0,1\}$ defines a concept $M_1(f)$.
Concept learning with membership queries for classes of threshold functions, 
$k$-threshold functions, and polytopes with vertices in $E_n^d$ corresponds to the problem of identifying 
geometric objects in $E_n^d$ with certain properties.  

From results of \cite{Hegedus} and \cite{Zolotykh} it follows that the learning complexity of the class of threshold functions $\mathfrak{T}(d,n)$ is $O\left(\frac{\log_2^{d-1}n}{\log\log_2n}\right)$.
In \cite{Maass} Maass and Bultman studied learning complexity of the class $k$-$HALFSPACES^2_{n,p}$, where $ 0 < p \leq \frac{\pi}{2}$.
The class $k$-$HALFSPACES^2_{n,p}$ is the subclass of $k$-threshold functions over $E_n^2$ with restrictions that for any $f$ in this 
subclass $P(f)$ has edges with length at least $16 \cdot \left\lceil \frac{1}{p} \right\rceil$ and an angle $\alpha$ between a pair of 
adjacent edges satisfies $p \leq \alpha \leq \pi - p$.
The learning algorithm proposed in \cite{Maass} for identification a function $f$ in $k$-$HALFSPACES^2_{n,p}$ requires a vertex of the polygon $P(f)$
as input and uses $O(k(\frac{1}{p} + \log n))$ membership queries.

Let $\mathcal{C}$ be a class of $\{0,1\}$-valued functions over the domain $X$ and $f\in \mathcal{C}$.
A \emph{teaching set} of a function $f$ with respect to $\mathcal{C}$ is a set of points $T\subseteq X$ such that
the only function in $\mathcal{C}$ which agrees with $f$ on $T$ is $f$ itself.
A teaching set $T$ is \emph{minimal} if no of its proper subset is teaching for $f$. 
Note that a teaching set of a function $f\in \mathfrak{T}(d,n,k)$ with 
respect to $\mathfrak{T}(d,n,*)$ is a teaching set with respect to $\mathfrak{T}(d,n,k)$.
A point $x \in X$ is called \emph{essential} for a function $f\in \mathcal{C}$ with respect to $\mathcal{C}$ 
if there exists a function $g \in \mathcal{C}$
such that $f(x) \neq g(x)$ and $f$ agrees with $g$ on $X \setminus \{x\}$.
Let us denote the set of essential points of a function $f$ with respect to a class $\mathcal{C}$ by $S(f,\mathcal{C})$ 
or by $S(f)$ when $\mathcal{C}$ is clear.
Let $S_{\nu}(f)=S(f)\cap M_{\nu}(f)$.
By $J(f,C)$ we denote the number of minimal teaching sets of a function $f$ with respect to a class $C$
and by $\sigma(f, \mathcal{C})$ the minimum cardinality of a teaching set of $f$ with respect to $\mathcal{C}$. 
The \emph{teaching dimension} of a class $\mathcal{C}$ is defined as
$$
\sigma(\mathcal{C}) = \max_{f \in \mathcal{C}} \sigma(f, \mathcal{C}).
$$

The main aim of a learning algorithm with membership queries is to find any teaching 
set of a target function $f$ with respect to a concept class $\mathcal{C}$.
The algorithm succeeds if it asked queries in all points of some teaching set of the function.
Therefore the teaching dimension of the class $\mathcal{C}$ is a lower bound on the learning complexity of this class.

It is known (see, for example, \cite{Brightwell} and \cite{Shevchenko}), that the set of essential points of a threshold function is a teaching set of this function. 
Together with the simple observation that any teaching set of a function should contains all its essential points, this imply that any threshold function have a unique minimal teaching set,
that is $J(f, \mathfrak{T}(d,n))=1$.
In addition, it follows from \cite{Zolotykh, Shevchenko, Shevchenko2, Shevchenko3} that for any fixed $d \geq 2$ 
$$
\sigma(\mathfrak{T}(d, n)) = \Theta(\log_2^{d-2}n) \quad (n \rightarrow \infty).
$$

In this paper we study combinatorial and structural properties of teaching sets of $k$-threshold functions for $k \geq 2$. 
In particular, we show that $2$-threshold functions
from $\mathfrak{T}(2,n,2)$, in contrast with threshold functions, can have more than one minimal teaching set with respect to $\mathfrak{T}(2,n,2)$. 
Moreover, we construct a sequence of functions from $\mathfrak{T}(2,n,2)$ for which number of minimal teaching sets grows as $\Omega(n^2)$. 
On the other hand, we show that any $k$-threshold function $f$ (or a polytope with vertices in $E_n^d$) has a unique minimal teaching set with respect 
to $\mathfrak{T}(d,n,*)$ coinciding with the set of essential points of $f$ with respect to $\mathfrak{T}(d,n,*)$. 
In addition, we give a general description of minimal teaching sets of such functions.
For functions in $\mathfrak{T}(2,n,*)$ we refine the given structure and derive a bound on the cardinality of the minimal teaching sets.

The organization of the paper is as follows. 
In Section 2, we consider essential points of an $k$-fold conjunction of an arbitrary $\{0,1\}$-valued functions $f_1,\dots,f_k$ and their connection with essential points of these functions. 
In the beginning of Section 3 we show that in general a $k$-threshold function can have more than one minimal teaching set. 
The main result of Subsection 3.1 (Theorem \ref{th:S_f}) states that a minimal teaching set of a $k$-threshold function with respect to $\mathfrak{T}(d,n,*)$ is unique and coincides with $S(f, \mathfrak{T}(d,n,*))$. 
The structure of $S(f,\mathfrak{T}(d,n,*))$ is given as well. 
In Subsection 3.2 we consider the class $\mathfrak{T}(2,n,*)$ and for a function $f$ in the class we prove an upper bound on the cardinality of $S(f, \mathfrak{T}(2,n,*))$. 
Finally, in Subsection 3.3 we consider functions in $\mathfrak{T}(2,n,2)$ with special properties and show that each of these functions has a minimal teaching set with cardinality at most 9 and there are functions with $\Omega(n^2)$ minimal teaching sets with respect to $\mathfrak{T}(2,n,2)$.

\section{The set of essential points of a $\{0,1\}$-valued functions conjunction}

Since a $k$-threshold function is a conjunction of $k$ threshold functions, 
it is interesting to investigate connection between essential points of threshold functions $f_1,\dots,f_k$ and essential points of their conjunction.
In this section we prove several propositions that establish this relationship. 
For a natural $k > 1$ and a class $\mathcal{C}$ of $\{0,1\}$-valued functions we denote by $\mathcal{C}^k$ the class of functions which can be presented as conjunction of $k$ functions from $\mathcal{C}$.
\begin{proposition}
\label{prop:statem1}
	Let $\mathcal{C}$ be a class of $\{0,1\}$-valued functions over  a domain $X$ and $f_1, \ldots, f_k \in \mathcal{C}$. Then for the function $f = f_1 \land \dots \land f_k$ the following inclusions hold:
	$$
		S_1(f_i, \mathcal{C}) \cap M_1(f) \subseteq S_1(f,\mathcal{C}^k)\qquad (i = 1, \ldots, k).
	$$
\end{proposition}
\begin{proof}
	Let $x \in S_1(f_i, \mathcal{C}) \cap M_1(f)$ for some $i \in \{1, \ldots, k\}$. Since $x$ is an
	essential point of $f_i$ and $f_i(x)=1$, there exists a function $f_i' \in \mathcal{C}$ which differs from $f_i$ in the unique point $x$. Denote by $f'$ the conjunction $f_1 \land \ldots \land f_{i-1} \land f_i' \land f_{i+1} \land \ldots \land f_k$. The function $f'$ belongs to the class $\mathcal{C}^k$ and differs from $f$ in the unique point $x$, namely
	$f'(x) = 0 \neq f(x)$. It means that $x$ is essential for $f$, i.e. $x \in S_1(f, \mathcal{C}^k)$.
\end{proof}

\begin{proposition}
\label{prop:statem2}
	Let $\mathcal{C}$ be a class of $\{0,1\}$-valued functions over a domain $X$ and $f_1, \ldots, f_k \in \mathcal{C}$. Then for the function $f = f_1 \land \dots \land f_k$ the following inclusions hold:
	$$
		S_0(f_i, \mathcal{C}) \cap \bigcap\limits_{j \neq i} M_1(f_j) \subseteq S_0(f, \mathcal{C}^k)\qquad (i = 1, \ldots, k).
	$$
\end{proposition}
\begin{proof}
	Let $x \in S_0(f_i, \mathcal{C}) \cap \bigcap\limits_{j \neq i} M_1(f_j)$ for some $i \in \{1, \ldots, k\}$.
	Since $x \in S_0(f_i)$, there exists a function $f_i' \in \mathcal{C}$ such that
	$f_i'(x) = 1$ and $f_i'(y)=f_i(y)$ for every $y \in X \setminus \{x\}$. Denote by $f'$ the conjunction $f_1 \land \ldots \land f_{i-1} \land f_i' \land f_{i+1} \land \ldots \land f_k$. The function $f'$ belongs to the class $\mathcal{C}^k$ and, since
	$x \in \bigcap\limits_{j \neq i} M_1(f_j)$, it differs from $f$ in the unique point $x$, namely $f'(x) = 1 \neq f(x)$. Therefore
	$x$ is essential for $f$ and $x \in S_0(f, \mathcal{C}^k)$. 
\end{proof}

\begin{proposition}
\label{prop:statem3}
	Let $\mathcal{C}$ be a class of $\{0,1\}$-valued functions over a domain $X$ and $f \in \mathcal{C}^k$. If there exists a unique set $f_1, \ldots, f_k \in \mathcal{C}$ such that $f = f_1 \land \dots \land f_k$, then
    $$
    S(f_i, \mathcal{C}) \subseteq \bigcap\limits_{j \neq i}M_1(f_j)\qquad (i = 1, \ldots, k).
    $$
\end{proposition}
\begin{proof}
	Suppose to the contrary that there exists $x \in X$ such that
    $x \in S(f_i, \mathcal{C})$ and $f_j(x) = 0$ for some distinct indices $i,j \in \{ 1, \ldots, k \}$.
    It means that $f(x)=0$.
    Since $x$ is essential for $f_i$, there exists a function $f_i' \in \mathcal{C}$ which differs from $f_i$ in the unique point $x$. 
    Clearly, $f_1 \land \ldots \land f_{i-1} \land f_i' \land f_{i+1} \land \ldots \land f_k = f$, which contradicts the uniqueness of the set $\{f_1,\dots,f_k\}$.
\end{proof}

\begin{corollary}\label{cor:col1}
	Let $\mathcal{C}$ be a class of $\{0,1\}$-valued functions over a domain $X$ and $f \in \mathcal{C}^k$. 
	If there exists a unique set $f_1, \ldots, f_k \in \mathcal{C}$ such that $f = f_1 \land \dots \land f_k$ then 
    $$
    \bigcup\limits_{i=1}^{k} S_\nu(f_i, \mathcal{C}) \subseteq S_\nu(f, \mathcal{C}^k)\qquad (\nu = 0, 1).
    $$
\end{corollary}
\begin{proof}
	Since the function $f$ satisfies the conditions of Proposition \ref{prop:statem3},
	$$
		S_1(f_i, \mathcal{C}) \subseteq M_1(f)\qquad (i = 1, \ldots, k)
	$$
	and
	$$
		S_0(f_i, \mathcal{C}) \subseteq \bigcap_{j \neq i} M_1(f_j)\qquad (i = 1, \ldots, k).
	$$
	By Propositions \ref{prop:statem1} and \ref{prop:statem2} we get
	$$
	\bigcup_{i=1}^k{S_\nu(f_i, \mathcal{C})} \subseteq S_\nu(f, \mathcal{C}^k).
	$$
\end{proof}
\section{Teaching sets of $k$-threshold functions}
Recall that the minimal teaching set of a threshold function is unique and equal to the set of its essential points. 
The situation becomes different for $k$-threshold functions when $k \geq 2$. 
We illustrate this difference in the following example.

\begin{example} Let $f$ be a function from $\mathfrak{T}(2, 4, 2)$ with
$$
M_1(f)=\{(1,2), (1,3), (2,2), (2,3)\}.
$$
The set of essential points $S(f)$ is
$$
	\{(1,1), (1,2), (2,1), (2,2), (0,3), (3,3)\}.
$$
This set is not a teaching set because there exists a function $g \in \mathfrak{T}(2,4,2)$ with $M_1(g)=\{(1,2), (2,2)\}$, which agrees with $f$ on $S(f)$ (see Fig. \ref{fig:2-threshold}).
Though if we add any of the two points $(1,3)$ or $(2,3)$ to $S(f)$, then we get a minimal teaching set of the function $f$ (see Fig. \ref{fig:2-threshold1}) with respect to 
$\mathfrak{T}(2, 4, 2)$, and therefore $J(f,\mathfrak{T}(2,4,2)) \geq 2$.
\end{example}

\begin{figure}[p]
    \centering
    \includegraphics[width=0.48\textwidth]{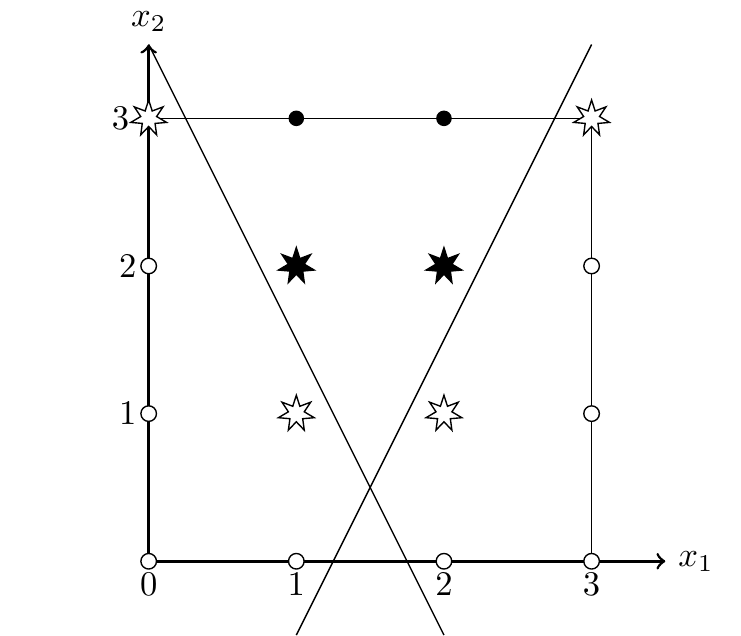}
   ~
    \includegraphics[width=0.48\textwidth]{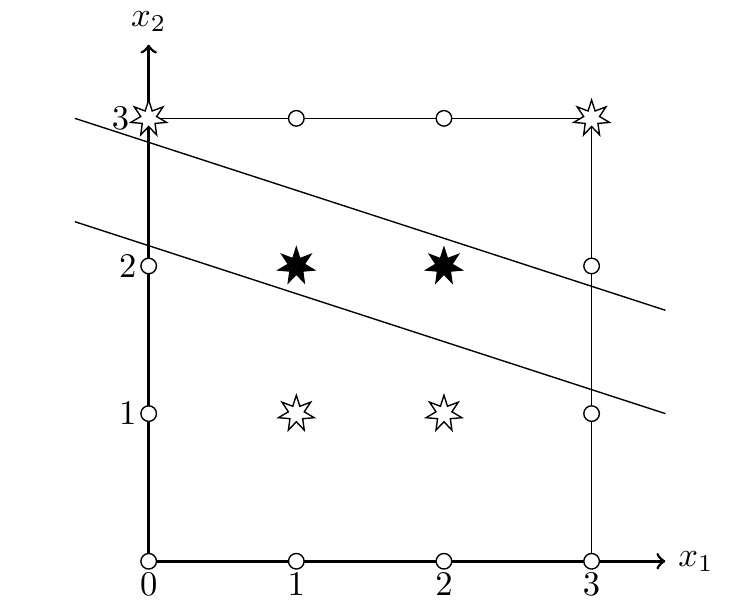}
    \caption{The stars denote the essential points. 
	The black elements denote the points from $M_1(f)$. 
	The empty elements denote the points from $M_0(f)$. 
	The functions $f$ (left plot) and $g$ (right plot) agree on $S(f)$.}
    \label{fig:2-threshold}
\end{figure}

\begin{figure}[p]
    \centering
    \includegraphics[width=0.48\textwidth]{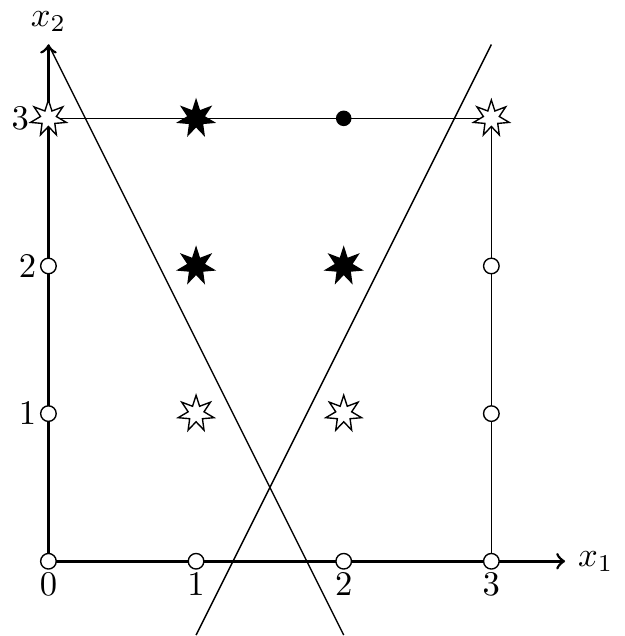}
   ~
    \includegraphics[width=0.48\textwidth]{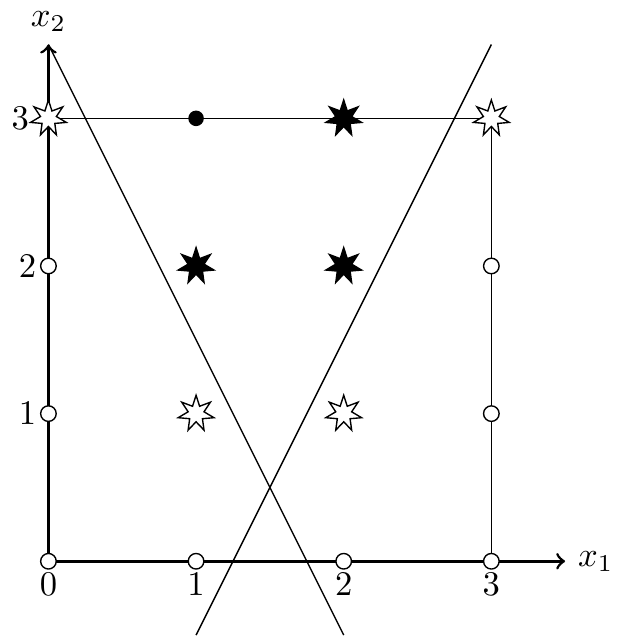}
    \caption{The stars denote the points of the minimal teaching sets $S(f) \cup \{1,3\}$ (left plot) and $S(f) \cup \{2,3\}$ (right plot).}
    \label{fig:2-threshold1}
\end{figure}

\subsection{Teaching sets for functions in $\mathfrak{T}(d,n,*)$}

In this section we prove that for $k \geq 2$ and $d \geq 2$ the teaching dimension of 
$\mathfrak{T}(d,n,k)$ is $n^d$.
Then we consider the class $\mathfrak{T}(d,n,*)$ and show that
for a function $f \in \mathfrak{T}(d,n,*)$ the set of its essential points with respect to $\mathfrak{T}(d,n,*)$ is also a teaching set, and therefore it is a unique minimal teaching set of $f$ with respect to $\mathfrak{T}(d,n,*)$.

\begin{lemma}
\label{lem:statem-line}
	Let $f:~E_n^d\to\{0,1\}$ be a function such that $1 \leq |\textup{Vert}(P(f))| \leq 2$  and $P(f) \cap M_0(f) = \emptyset$. 
	Then $f \in \mathfrak{T}(d,n,k)$ for any $k \geq 2$.
\end{lemma}
\begin{proof}
	It is sufficient to show that $f$ is a $2$-threshold function. 
	Let $x$ and $y$ be the two vertices of $P(f)$. Note that if $|M_1(f)| = 1$, then
	$x=y$. 
	
	Clearly, it is possible to choose two parallel hyperplanes $H'$ and $H''$ sufficiently
	close to each other such that $E_n^d \cap H' = \{x\}$, $E_n^d \cap H'' = \{y\}$, 
	and there are no points between $H'$ and $H''$ in $E_n^d \setminus M_1(f)$. These
	hyperplanes can be used to define a 2-threshold function, that coincides with $f$.
\end{proof}

In \cite{Angluin} it was established that the teaching dimension of a class containing the empty set and $N$ singleton sets is at least $N$. 
This result and Lemma \ref{lem:statem-line} give us the teaching dimension for $\mathfrak{T}(d,n,k)$, where $k \geq 2$:
\begin{corollary}
	$\sigma(\mathfrak{T}(d,n,k)) = n^d$ for every $k \geq 2$.
\end{corollary}

For a polytope $P$ denote by $\textup{Vert}(P)$ the set of vertices of $P$, by $B(P)$ the set of integer points on the border of $P$ and by $\textup{Int}(P)$ the set of internal integer points of $P$.
For $f \in \mathfrak{T}(d,n,*)$ denote by $D(f)$ the set $\{x \in M_0(f): \textup{Conv}(P(f) \cup \{x\}) \cap M_0(f) = \{x\} \}$.

\begin{theorem}
\label{th:S_f}
Let $f \in \mathfrak{T} (d,n,*)$, $d \geq 2$, $n \geq 2$. Then 
$$
S(f, \mathfrak{T} (d,n,*)) = 
\begin{cases}
	E_n^d,& M_1(f) = \emptyset; \\
	\textup{Vert}(P(f)) \cup D(f),& M_1(f) \neq \emptyset;
\end{cases}
$$ 
and $S(f, \mathfrak{T} (d,n,*))$ is a unique minimal teaching set of $f$.

\end{theorem}
\begin{proof}
	If $M_1(f) = \emptyset$, then $f \equiv 0$, and therefore $S(f) = E_n^d$. Clearly, in this case $S(f)$ is a unique minimal teaching set for $f$.
    
	Now let $M_1(f) \neq \emptyset$. We split the proof of this case into two parts. 
	At first we show that all points from $\textup{Vert}(P(f)) \cup D(f)$ are essential, and
	then we prove that this set is a unique minimal teaching set.
\begin{enumerate}
\item
Let $f': E_n^d \rightarrow \{0,1\}$ be a function which 
differs from $f$ in a unique point $x \in \textup{Vert}(P(f))$.
Obviously $P(f') \cap M_0(f') = \emptyset$
and  $f'$ belongs to $\mathfrak{T} (d,n,*)$. Therefore $x$
is essential for $f$ with respect to $\mathfrak{T} (d,n,*)$.
Now let $f': E_n^d \rightarrow \{0,1\}$ be a function which 
differs from $f$ in a  unique point $x \in D(f)$. 
The choice of $x$ implies that the function $f'$ belongs to $\mathfrak{T} (d,n,*)$ and hence $x$ is essential point of $f$ with respect to $\mathfrak{T} (d,n,*)$.

\item
Since $f$ belongs to the class $\mathfrak{T} (d,n,*)$, knowing values of the function in $\textup{Vert}(P(f))$ is sufficient to recover $f$ in $M_1(f)$.
Further, for every point 
$x \in M_0(f)$ such that $|\textup{Conv}(P(f) \cup \{x\}) \cap M_0(f)| > 1$ the set $\textup{Conv}(P(f) \cup \{x\}) \cap M_0(f)$ necessarily contains at least one point from $D(f)$. 
Therefore, to recover $f$ in $M_0(f)$ it is sufficient to know the function values in points from $D(f)$ and $\textup{Vert}(P(f))$.
This leads us to a conclusion that $\textup{Vert}(P(f)) \cup D(f)$ is a teaching set. 
Moreover, since all points in this set are essential and any teaching set contains all essential points, we conclude that $\textup{Vert}(P(f)) \cup D(f)$ is a unique minimal teaching set and coincides with $S(f)$.
\end{enumerate}
\end{proof}

\begin{lemma}
\label{lem:m1}
	Let $f \in \mathfrak{T}(d,n,k), d \geq 2, k \geq 2$ and $M_1(f) = \{ x' \}$. Then
	$$
		S(f, \mathfrak{T}(d,n,k)) = \{ x' \} \cup \{ x \in E_n^d : \textup{GCD}(|x_1 - x_1'|, \ldots, |x_d-x_d'|) = 1 \},
	$$
	and $S(f, \mathfrak{T}(d,n,k))$ is a unique teaching set of $f$ with respect to $\mathfrak{T}(d,n,k)$ and
	$$
		|S(f, \mathfrak{T}(d,n,k))| = \Theta(n^d).
	$$
\end{lemma}
\begin{proof}

	Let $S = \{x \in E_n^d: \text{GCD}(|x_1 - x'_1|, \dots, |x_d - x'_d|) = 1\}$.
	For any $x \in S$ the segment $x'x$ does not contain other points from $E_n^d$ except $x$ and $x'$, that is $x'x \cap E_n^d = \{x', x\}$. 
	Then, according to Lemma \ref{lem:statem-line}, a function $g:~ E_n^d \rightarrow \{0, 1\}$ with $M_1(g) = \{x', x\}$ belongs to the class $\mathfrak{T}(d, n, k)$ for any $k \geq 2$.
	Since $x$ distinguishes $g$ from $f$, it is an essential point for the both functions.
	Therefore all points of $S$ are essential for $f$.
	On the other hand, $S \cup \{x'\}$ is a teaching set for $f$ because for any point $y \in E_n^d \setminus \{ S \cup \{x'\}\}$ there exists a point $y' \in S$ such that $y,y',x'$ are collinear and $y'$ is between $y$ and $x'$.
	
	Let $\varphi(i)$ be the Euler totient function. It is well known (see, for example, \cite{Euler}) that 
	$$
	\sum_{i \leq n}{\varphi(i)} = \frac{3}{\pi^2}n^2 + O(n \ln{n}).
	$$
	
	Using this formula we can get a lower bound on the cardinality of the minimal teaching set:
	$$
		|S \cup \{x'\}| = \left|\{x = (x_1,\dots,x_d) \in E_n^d: \text{GCD}(|x_1 - x'_1|, \dots, |x_d - x'_d|) = 1\}\right| + 1 \geq
	$$
	$$
		\geq \left|\{x = (x_1,\dots,x_d) \in E_n^d: x_3= \dots = x_d=0,\text{GCD}(|x_1 - x'_1|, |x_2 - x'_2|) = 1\}\right| n^{d-2} \geq
	$$
	$$
		\geq \left(\sum_{i \leq n/2}{\varphi(i)}\right) n^{d-2} = \left(\frac{3}{\pi^2}\left(\frac{n}{2}\right)^2 + O\left(\frac{n}{2} \ln{\frac{n}{2}}\right)\right) n^{d-2} = \Omega(n^d).
	$$
	
	Since $|E_n^d|=n^d$, this lower bound matches a trivial upper bound, and therefore $|S(f, \mathfrak{T}(d,n,k))| = \Theta(n^d)$.
\end{proof}

\subsection{Teaching sets of functions in $\mathfrak{T}(2,n,*)$}

In the previous section we proved that for a function from $\mathfrak{T}(d,n,*), d \geq 2$ the set of its essential points is also the unique minimal teaching set.
In this section we consider the class $\mathfrak{T}(2,n,*)$ and describe the structure of the set of essential points for a function in this class. 
We also give an upper bound on the cardinality of this set.

Let us consider an arbitrary function $f \in \mathfrak{T} (2,n,*)$. 
Note that $P(f)$ can be the empty set, a point, a segment or a polygon. 
Let $P(f)$ be a segment or a polygon, that is $|M_1(f)| > 1$, and let $a_{1}x_1 + a_{2}x_2 = a_{0}$ be the edge equality for an edge $e$ of $P(f)$. 
Without loss of generality we may assume that $\text{GCD}(a_{1}, a_{2}) = 1$. 
Denote by \emph{edge inequality} for edge $e$ inequality $a_{1}x_1 + a_{2}x_2 \leq a_{0}$ or/and $a_{1}x_1 + a_{2}x_2 \geq a_{0}$ if it is true for all points of $P(f)$.
Note that if $P(f)$ is a segment, then it has one edge but two edge inequalities corresponding to the edge.
If $P(f)$ is a polygon, then it has exactly one edge inequality for each edge.
Hence the number of edge inequalities for $P(f)$ is equal to the number of its vertices.

Let $f$ be a function from $\mathfrak{T} (2,n,*)$ with $|M_1(f)| > 1$ and let
$$
a_{i1}x_1 + a_{i2}x_2 \leq a_{i0}, \quad  i=1,\dots,|\textup{Vert}(P(f))| 
$$ 
be edge inequalities for $P(f)$. The \emph{extended edge inequality} for an edge $e$ of $P(f)$ is $a_1x_1 + a_2x_2 \leq a_0 + 1$, where $a_1x_1 + a_2x_2 \leq a_0$ is the corresponding edge inequality for $e$.
By $P'(f)$ we denote the following extension of $P(f)$
\begin{equation}\label{eq:deltaP}
	\{x=(x_1,x_2) : a_{i1}x_1 + a_{i2}x_2 \leq a_{i0} + 1, \quad  i=1,\dots,|\textup{Vert}(P(f))|\}.
\end{equation}
We also let
$$
\Delta P(f) = P'(f) \setminus P(f).
$$
It follows from the definition that $P'(f)$ contains $P(f)$, and for every straight line $l'$ containing an edge of $P'(f)$ there exists an edge in $P(f)$ belonging to the closest parallel to the $l'$ straight line which contains integer points.

If $P$ is a polygon then denote by $\mathcal{P}(P)$ the perimeter of $P$, 
by $\mathcal{S}(P)$ the area of $P$ and by $q_{min}(P)$ the minimum angle between neighboring edges of $P$. 

The next proposition uses the Pick's formula (see \cite{Trainin}) for the area of a convex polygon $P$ with integer vertices:
$$
\mathcal{S}(P) = \textup{Int}(P) + \frac{B(P)}{2} - 1.
$$

\begin{proposition}\label{prop:deltaP}
Let $f \in \mathfrak{T} (2,n,*)$ and $\mathcal{S}(P(f)) > 0$. Then $D(f)= \Delta P(f) \cap M_0(f)$.
\end{proposition}
\begin{proof}
Note that by construction all integer points of $\Delta P(f)$ lie on the border of $P'(f)$, which implies that $\Delta P(f) \cap M_0(f) \subseteq D(f)$.
Consider $x=(x_1,x_2)$ such that $|\textup{Conv}(P(f)) \cup \{x\}) \cap M_0(f)| = 1$. 
To show that $x \in \Delta P(f)$ it is sufficient to prove that $x$ is a solution of the system of inequalities (\ref{eq:khalfspace1}), that is each extended edge inequality for $P(f)$ holds true for $x$.
Obviously, if an edge inequality is true for $x$, then the corresponding extended edge inequality is also true.
Let $e$ be the edge whose edge inequality is false for $x$, that is $a_1x_1+a_2x_2 > a_0$.
All integer points of the triangle $Tr=\textup{Conv}(e \cup \{x\})$ belongs to $e \cup \{x\}$.
Since $Tr$ has integer vertices, its area can be calculated by the Pick's formula: 
$$
\mathcal{S}(Tr) = \frac {|(e \cup \{x\}) \cap E_n^2|}{2} - 1 = \frac {|e \cap E_n^2| - 1}{2}.
$$
Comparing resulting equation with the classical triangle area formula $\mathcal{S}(Tr) = \frac{l(e)h_x}{2}$ we conclude that
$$
h_x = \frac{|e \cap E_n^2| - 1}{l(e)},
$$
where $h_x$ is the distance between point $x$ and the line containing $e$.

Now consider an integer point $y=(y_1,y_2)$ for which $a_1y_1 + a_2y_2 = a_0 + 1$.
Using the same arguments it is easy to show that
$$
h_y = \frac{|e \cap E_n^2| - 1}{l(e)}.
$$
Hence, $x$ and all integer points of the line $a_1y_1 + a_2y_2 = a_0 + 1$ have the same distance to the line containing $e$.
It means that $a_1x_1 + a_2x_2 = a_0 + 1$, that is the extended edge inequality for $e$ is true for $x$ and $x$ belongs to $P(f)$, therefore 
$D(f) \subseteq \Delta P(f) \cap M_0(f)$.
\end{proof}

\begin{corollary}
\label{corollary:S*}
Let $f \in \mathfrak{T} (2,n,*)$ and $\mathcal{S}(P(f)) > 0$. 
Then
$$
S(f, \mathfrak{T}(2,n,*)) = (\Delta P(f) \cap M_0(f)) \cup \textup{Vert}(P(f)).
$$ 
\end{corollary}


The next lemma establishes relationship between the perimeters of $P(f)$ and $P'(f)$ to help us to estimate the cardinality of the set of essential points of a function from $\mathfrak{T}(2,n,*)$.

\begin{lemma}\label{lem:perimeterPP'}
Let $f \in \mathfrak{T} (2,n,*)$ and $\mathcal{S}(P(f)) > 0$. Then 
$$
\mathcal{P}(P'(f)) \leq \mathcal{P}(P(f)) + 2 \sum_{i=1}^{|\textup{Vert}(P(f))|} \cot \frac{q_{i}(P(f))}{2},
$$
where $q_i(P(f))$ for $i \in \{1,\dots,|\textup{Vert}(P(f))|\}$ are the angles between neighboring edges of $P(f)$.
\end{lemma}
\begin{proof}
Denote by $P''$ the set of points satisfying such a condition that if an edge inequality is false for a point, then the distance between the point and the straight line containing the corresponding edge is at most 1.
Note that points of $P'(f)$ also satisfy the specified condition, so $P'(f) \subseteq P''$ and, consequently, $\mathcal{P}(P'(f)) \leq \mathcal{P}(P'')$ (see Fig. \ref{fig:PP''}).
Further, $P''$ is a convex polygon with $|\textup{Vert}(P(f))|$ vertices, and each edge $e''$ of $P''$ is parallel to some edge $e$ of $P(f)$ and is at distance 1 from the line containing $e$.
Let $q_i, q_{i+1}$ for some $i \in \{1, \dots, |\textup{Vert}(P(f))| - 1\}$ be the angles between $e$ and its neighboring edges in $P(f)$.
By construction of $P''$ we have:
$$
l(e'') = l(e) + \cot \frac{q_i}{2} + \cot \frac{q_{i+1}}{2}.
$$
Summing up the lengths of all edges of $P''$ we have:
$$
\mathcal{P}(P'(f)) \leq \mathcal{P}(P'') = \mathcal{P}(P(f)) + 2 \sum_{i=1}^{|\textup{Vert}(P(f))|} \cot \frac{q_{i}(P(f))}{2}.
$$
\end{proof}

\begin{figure}[p]
    \centering
    \includegraphics[width=0.48\textwidth]{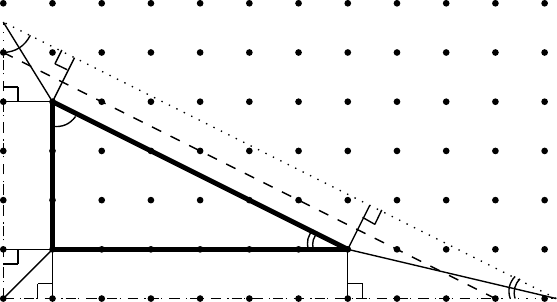}
    \caption{$P(f)$ (bold solid triangle), $P'(f)$ (dashed triangle) and $P''$ (dotted triangle).}
    \label{fig:PP''}
\end{figure}

\begin{theorem}
\label{theorem:statem-P}
Let $f \in \mathfrak{T} (2,n,*)$ and $\mathcal{S}(P(f)) > 0$. Then
$$
\left|S(f, \mathfrak{T}(2,n,*))\right| = O\left(\min\left(n, \mathcal{P}(P(f)) + \frac{1}{q_{min}(P(f))}\right)\right).
$$
\end{theorem}
\begin{proof}

By Corollary \ref{corollary:S*} we have $S(f, \mathfrak{T}(2,n,*)) = (\Delta P(f) \cap M_0(f)) \cup \textup{Vert}(P(f))$.
Since every point of $S(f, \mathfrak{T}(2,n,*))$ is integer and either belongs to
the border of $P(f)$ or to the border of $P'(f)$, the cardinality of 
$S(f, \mathfrak{T}(2,n,*))$ can be bounded from above by the sum of the
perimeters $\mathcal{P}(P(f))$ and $\mathcal{P}(P'(f))$.
	So we have:
	$$
	|S(f, \mathfrak{T}(2,n,*))| \leq \mathcal{P}(P(f)) + \mathcal{P}(P'(f)) \leq  
	2\mathcal{P}(P(f)) + \sum_{i=1}^{|\textup{Vert}(P(f))|}{2\cot{\frac{q_i(P(f))}{2}}},
    $$
where $q_i$ for $i \in \{1,\dots,|\textup{Vert}(P(f))|\}$ are the angles between neighboring edges of $P(f)$.
    
As the number of integer vertices of a convex polygon is not more than the perimeter of the polygon, we have $|\textup{Vert}(P(f))| \leq \mathcal{P}(P(f))$. Obviously, only 2 angles of a convex polygon can be less than $\frac{\pi}{3}$. Therefore	
  	$$
	2\mathcal{P}(P(f)) + \sum_{i=1}^{|\textup{Vert}(P(f))|}{2\cot{\frac{q_i(P(f))}{2}}} \leq
    $$
	$$
	\leq 2\mathcal{P}(P(f)) + 4\cot{\frac{q_{min(P(f))}}{2}} + \sqrt{3}(\mathcal{P}(P(f)) - 2).
    $$  
    Since $0 \leq \frac{q_{min}(P(f))}{2} < \frac{\pi}{2}$, we can conclude:
    $$
    \cot{\frac{q_{min}(P(f))}{2}} \leq \frac{1}{\sin{\frac{q_{min(P(f))}}{2}}} = O\left(\frac{1}{q_{min}(P(f))}\right).
    $$
    Therefore
    $$
    |S(f, \mathfrak{T}(2,n,*))| = O\left(\mathcal{P}(P(f)) + \frac{1}{q_{min}(P(f))}\right).
	$$
    Finally, since
    $$
    \mathcal{P}(P(f)) + \mathcal{P}(P'(f)) = O(n),
    $$
    we conclude:
    
    $$
    |S(f, \mathfrak{T}(2,n,*))| = O\left(\min\left(n, \mathcal{P}(P(f)) + \frac{1}{q_{min}(P(f))}\right)\right).
	$$
\end{proof}

\begin{example}
Consider a function $f \in \mathfrak{T}(2,12,*)$ (see Fig. \ref{fig:P}).
The gray set is $\Delta P(f)$. 
The black stars are the points from $\textup{Vert}(P(f))$
and the white stars are the points from $\Delta P(f) \cap M_0(f)$.
\end{example}

\begin{figure}[p]
    \centering
    \includegraphics[width=0.75\textwidth]{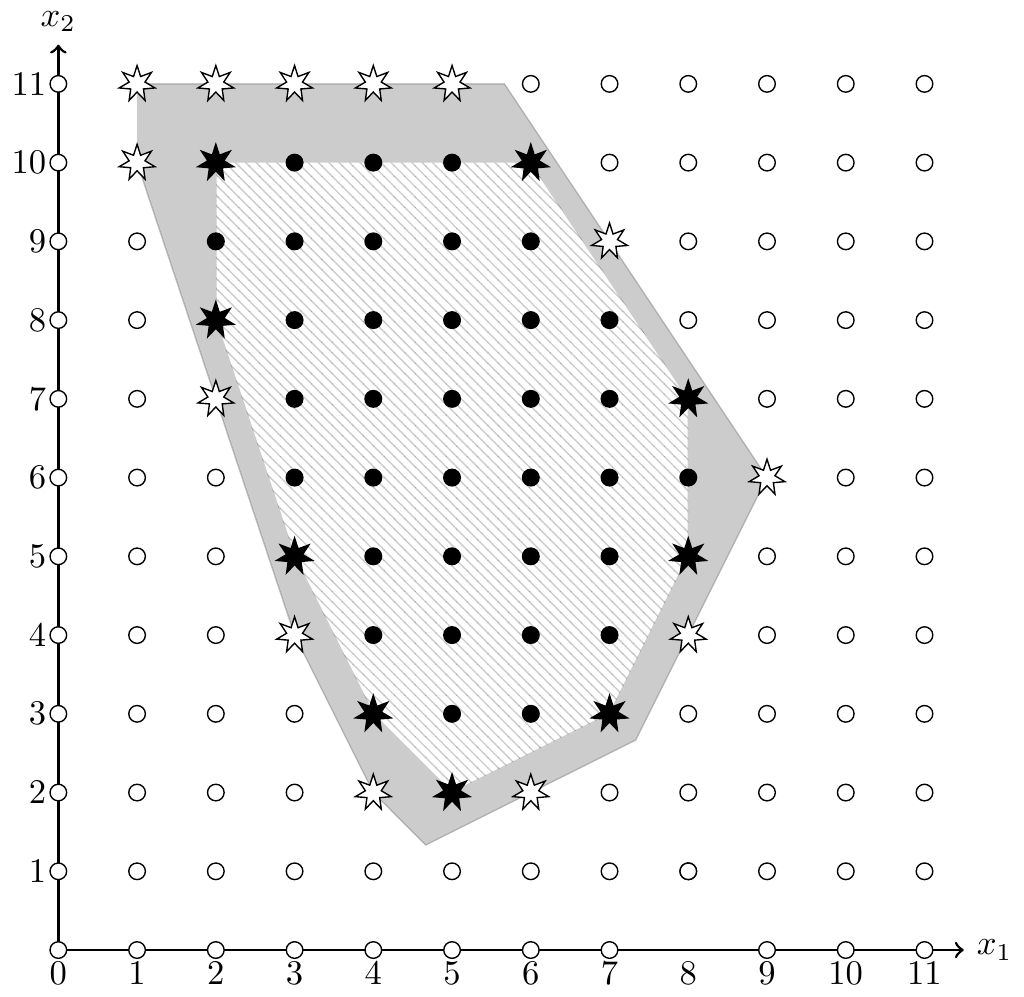}
    \caption{The gray set is $\Delta P(f)$, the stripped area is $P(f)$, and the union of both of them is $P'(f)$.}
    \label{fig:P}
\end{figure}

\begin{proposition} \label{prop:P}
Let $f \in \mathfrak{T} (2,n,*)$ and $M_1(f) > 1$. Then $f$ is a $|\textup{Vert}(P(f))|$-threshold function and the sets of essential points of $f$ with respect to $\mathfrak{T} (2,n,*)$ and with respect to $\mathfrak{T} (2,n, |\textup{Vert}(P(f))| + 1)$ coincide.
\end{proposition}
\begin{proof}
Lemma \ref{lem:statem-line} shows that functions $f$ with $|\Vertic(P(f))| = 2$ are $2$-threshold.
Let $|\Vertic(P(f))| > 2$.
The polygon $P(f)$ is a solution of a system of $|\textup{Vert}(P(f))|$ inequalities, and therefore $f$ is a $|\textup{Vert}(P(f))|$-threshold.
For any $x \in \Vertic(P(f))$ we can add one inequality to the system (\ref{eq:khalfspace1}) to get a function $f' \in \mathfrak{T}(2,n,|\textup{Vert}(P(f))| + 1)$ such that $M_1(f') = M_1(f) \setminus \{x\}$, hence $\textup{Vert}(P(f)) \subseteq S(f, \mathfrak{T}(2, n, |\textup{Vert}(P(f))| + 1)$.

Now, consider an arbitrary point $x \in D(f)$ and a function $f' \in \mathfrak{T}(2,n,*)$ with $M_1(f) = \textup{Conv}(P(f) \cup \{x\}) \cap E_n^2$. 
Obviously, $|\textup{Vert}(P(f'))| \leq |\textup{Vert}(P(f))| + 1$ and $f'$ is a $(|\textup{Vert}(P(f))| + 1)$-threshold function. 
The functions $f$ and $f'$ differ in the unique point $x$ and belong to the classes of $|\textup{Vert}(P(f'))|$-threshold and $(|\textup{Vert}(P(f))|+1)$-threshold functions, respectively. 
Therefore $x \in S(f, \mathfrak{T}(2, n, |\textup{Vert}(P(f))| + 1))$ and $D(f) \subseteq S(f, \mathfrak{T}(2, n, |\textup{Vert}(P(f))| + 1))$. 
According to Theorem \ref{th:S_f} the sets of essential points of $f$ with respect to $\mathfrak{T} (2,n,*)$ and with respect to $\mathfrak{T} (2,n, |\textup{Vert}(P(f))| + 1)$ coincide.
\end{proof}

\begin{example}
Consider a function $f \in \mathfrak{T}(2,n,4)$ with $ M_1(f)=\{(1,1),(1,2),(2,1)\}$ (see Fig. \ref{fig:fm}). 
We have $\textup{Vert}(P(f))= \{(1,1), (1,2), (2,1)\}$ and $f$ is a $3$-threshold function. 
Further, $\Delta P(f) \cap E_4^2 = \{(0,0),(1,0),(2,0),(3,0),(0,1),(3,1),(0,2),(2,2),(3,2), (0,3),(1,3)\}$, and hence $S(f,\mathfrak{T}(2,n,*)) = E_4^2 \setminus \{(3,2),(2,3),(3,3)\}$ = $S(f, \mathfrak{T}(2,n,4))$.
\end{example}

\begin{figure}[p]
    \centering
    \includegraphics[width=0.48\textwidth]{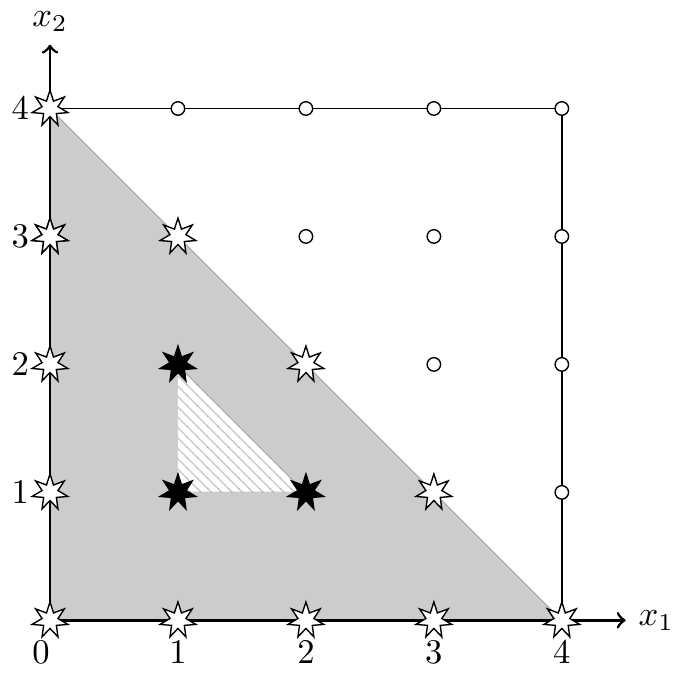}
    \caption{The gray shape is $\Delta P(f)$, the stripped area is $P(f)$.}
    \label{fig:fm}
\end{figure}

\subsection{The teaching set of functions from $\mathfrak{T} (2,n,2)$ with a unique defining set of threshold functions}

In this section we consider the subset of $2$-threshold functions over $E_n^2$, for which the cardinality of minimal teaching set can be bounded by a constant.
Also we show that for such $2$-threshold functions the number of minimal teaching sets can grow as $\Omega(n^2)$. 

Let $f \in \mathfrak{T}(2,n)$ and let $a_0, a_1, a_2$ be real numbers which are not all zero. 
We call the line $a_1x_1 + a_2x_2 = a_0$ an  \emph{$i$-separation line} (or just \emph{separation line}) of $f$ if there exists $i \in \{0,1\}$ such that
$$
x=(x_1,x_2) \in M_i(f) \Longleftrightarrow a_1x_1 + a_2x_2 \leq a_0.
$$
For example, the equality corresponding to a threshold inequality of $f$ defines a 1-separation line of $f$.
Let us prove some properties of separation lines of threshold functions.

It is known \cite{Zolotykh2} that $|S(g)| \in \{ 3, 4 \}$ and $|S_1(g)|,|S_0(g)| \in \{ 1, 2 \}$ for any $g \in \mathfrak{T}(2,n)$ and the $1$-valued essential points of $g$ are adjacent vertices of $P(g)$.

\begin{proposition}\label{prop:sepline1}
Let $f \in \mathfrak{T}(2,n)$. For any $i \in \{0,1\}$ there exists an $i$-separation line of $f$ which contains all points of $S_i(f)$.
\end{proposition}
\begin{proof}
Clearly, it is enough to prove the proposition for $i=1$. 
Denote by $l$ some $1$-separation line of $f$ which does not contain integer points and let $x \in S_1(f)$.
There exists a function $g \in \mathfrak{T}(2,n)$ such that $x$ distinguishes $f$ from $g$, that is $f(y)=g(y)$ for all $y \in E_n^2 \setminus \{x\}$ and $g(x) = 0$.
Denote by $l'$ some $1$-separation line for $g$ which also does not contain integer points.
If $l$ and $l'$ are parallel lines then $x$ lies between them. 
In this case we can pass through $x$ a parallel to $l$ and $l'$ straight line $l''$ which will be a $1$-separate line of $f$.
If $l$ and $l'$ intersect in some point $y$, then the straight line $l''$ which intersects $x$ and $y$ is a $1$-separation line of $f$.
Thus, for any essential point there exists a separation line which intersects $x$ and does not contain any other integer points.
This proves the proposition for $|S_1(f)| = 1$.

Now let $|S_1(f)| = 2$ and $S_1(f) = \{x,y\}$.
There exist functions $g_x,g_y \in \mathfrak{T}(2,n)$ such that $f(z) = g_j(z)$ for all $z \in E_n^2 \setminus \{j\}$ and $g_j(j)=0$, where $j \in \{x,y\}$.
Denote by $l_j$ a $1$-separation line for $g_j$ which does not contain integer points except point $j$.
By construction, sets $M_0(g_x) \cap M_0(g_y)$ and $M_1(g_x) \cap M_1(g_y)$ are separated by the straight line $l'$ containing $x$ and $y$. 
Since $M_0(g_x) \cap M_0(g_y) = M_0(f)$ and $M_1(g_x) \cap M_1(g_y)  = M_1(f) \setminus l'$, the line $l'$ is a $1$-separation line. 
\end{proof}

\begin{proposition}\label{prop:sepline2}
Let $f \in \mathfrak{T}(2,n)$ and $l$ is an $i$-separation line for $f$ for some $i \in \{0,1\}$. 
Then $\textup{Vert}(\textup{Conv}(l \cap E_n^2)) \subseteq S_i(f)$.
\end{proposition}
\begin{proof}
Assume without loss of generality that $i=1$ and $l$ is a $1$-separation line.
If $l \cap E_n^2 = \emptyset$, then the proposition is obvious.
Suppose $l \cap E_n^2 = \{x\}$, that is $l$ intersects $E_n^2$ in exactly one point $x$.
It is easy to see that $l$ is also a $0$-separation line for a function $g \in \mathfrak{T}(2,n)$ which coincides with $f$ on $E_n^2 \setminus \{x\}$ and $g(x)=0$, therefore $x$ is an essential point for $f$.
Since $l$ is a $1$-separation line for $f$ and $x \in l$, we conclude that $x \in S_1(f)$.

Now suppose that $|l \cap E_n^2| > 1$ and $\textup{Vert}(\textup{Conv}(l \cap E_n^2)) = \{x, y\}$.
We can turn $l$ around $x$ on a small angle (to not intersect any other integer points) in such a direction that $y$ would be on the same halfspace from the line as other points of $M_1(f)$. 
New line $l'$ will be $1$-separation line for $f$ containing exactly one integer point $x$, and, as we showed above, $x \in S_1(f)$.
The same arguments are true for $y$, that is $y \in S_1(f)$.
\end{proof}
\begin{proposition}\label{prop:sepline3}
Let $f \in \mathfrak{T}(2,n)$.
The sets $S_0(f)$ and $S_1(f)$ belong to the parallel separation lines and there is no integer points between the lines.
\end{proposition}
\begin{proof}
Assume without loss of generality that $|S_1(f)| = 2$ and $S_1(f) = \{x,y\}$. 
By proposition \ref{prop:sepline1} the line $l$ containing $S_1(f)$ is a $1$-separation line for $f$.
We can make a translation of $l$ in direction to $M_0(f)$ to the closest line $l'$ which intersects at least one point from $M_0(f)$.
If $|S_0(f)| = 1$ and $S_0(f) = \{z\}$, then $z \in l'$ and the proposition holds.
Let $|S_0(f)| = 2$ and $S_0(f) = \{z, u\}$.
Note that triangles $\triangle xyz$ and $\triangle xyu$ contain no other integer points, except the vertices and the points on the segment $xy$.
By the Pick's formula both triangles have the same area.
It means that both $z$ and $u$ are at the same distance from $l$ and lie on the line $l'$. 
\end{proof}

\begin{theorem}
\label{theorem:statem-9}
Let $f \in \mathfrak{T}(2, n, 2)$ and $M_1(f) \cap B(\textup{Conv}(E_n^2)) \neq \emptyset$, and let some set of threshold functions 
$\{f_1,f_2 \}$ defining f satisfies the following system:
\begin{equation}
\label{eq:S_M_empty}
\begin{cases}
	S(f_1) \cap M_0(f_2) = \emptyset; \\
	S(f_2) \cap M_0(f_1) = \emptyset.
\end{cases}
\end{equation}
Then $\{f_1, f_2\}$ is a unique defining set of $f$ and
$$
	\sigma(f, \mathfrak{T}(2, n, 2)) \leq 9.
$$
\end{theorem}
\begin{proof}
Note that
$$
	B(\textup{Conv}(E_n^2)) = \{x \in E_n^2: x_1 = 0 \vee  x_2 = 0 \vee x_1 = n-1 \vee x_2 = n-1\}.
$$
We consider two cases depending on the cardinalities of $S_0(f_1)$, $S_0(f_2)$.

Let $|S_0(f_i)| = 1$ for some $i \in \{1,2\}$. Assume, without loss of generality, that $|S_0(f_1)| = 1$, that is $S_0(f_1)=\{ u \}$. Then $|S_1(f_1)| =2$ and $S_1(f_1) = \{ v_1,v_2 \}$.
Consider an arbitrary function $f' \in \mathfrak{T}(2,n,2)$ which agrees with $f$ on $S(f_1) \cup S(f_2)$ and some of its defining set of threshold functions $F' = \{f_1', f_2'\}$.
From the first equation of the system (\ref{eq:S_M_empty}) it follows that $f_1(x) = f(x) = f'(x)$ for every $x \in S(f_1)$.
Hence one of the functions from $F'$, say $f_1'$, should take 
the value $0$ on $u$ and the value $1$ on $v_1$ and $v_2$, therefore
\begin{equation}\label{eq:f1}
	f_1' = f_1.
\end{equation}

From the second equation of the system (\ref{eq:S_M_empty}) we have $f_2(x) = f(x) = f'(x)$ for every $x \in S(f_2)$. 
This together with (\ref{eq:f1}) imply that 
$f_2'$ agrees with $f_2$ on $S(f_2)$, and therefore
$f_2' = f_2$.
We showed that $\{f_1, f_2\}=F'$, and hence $f'$ coincides with $f$ and $\{f_1, f_2\}$ is a unique defining set for $f$.
Moreover, $S(f_1) \cup S(f_2)$ is a teaching set of $f$ and $|S(f_1) \cup S(f_2)| \leq 7$. 

Now suppose that $|S_0(f_1)|=|S_0(f_2)|=2$, that is $S_0(f_1)=\{u_1, u_2\}$, $S_0(f_2)=\{v_1, v_2\}$.
Denote by $G \subseteq \mathfrak{T}(2, n, 2)$
a set of 2-threshold functions, which agree with $f$ on $S(f_1) \cup S(f_2)$.
From the conditions of the theorem it follows
that $S_0(f_i) \in M_1(f_j)$ for $i \neq j$.
Note that $S_0(f_1) \cup S_0(f_2)$ is a set of vertices of a convex quadrilateral $P=(u_1, u_2, v_1, v_2)$, 
and for each of the threshold functions $\{f_1,f_2\}$ vertices from its teaching set are neighboring (see Fig. \ref{fig:p9}). 
This implies that $G$ is the union of two sets:
$$
	G_1 = \{ g \, | \, g \in G, \exists g_1,g_2 \in \mathfrak{T}(2,n): g = g_1 \land g_2, \{ u_1,u_2 \} \subseteq M_0(g_1),
    \text{ and } \{ v_1,v_2 \} \subseteq
    M_0(g_2)\}
$$
and
$$
	G_2 = \{ g \, | \, g \in G, \exists g_1,g_2 \in \mathfrak{T}(2,n): g = g_1 \land g_2, \{ u_1,v_2 \} 
    \subseteq M_0(g_1), \text{ and } \{ u_2,v_1 \} \subseteq M_0(g_2)\}.
$$
Applying the same arguments as in the previous case where $|S_0(f_i)| = 1$ it can be shown that $G_1 = \{ f \}$.
Now, to prove the uniqueness of a defining set of $f$ it is sufficient to demonstrate that $f \notin G_2$.
To this end, we first show that
\begin{equation}\label{eq:Int_P}
M_1(f) \cap \bigcup_{g \in G_2}{M_1(g)} \subseteq \textup{Int}(P).
\end{equation}
By Proposition \ref{prop:sepline1} the line $l$ containing $u_1,u_2$ and the line $l'$ containing $v_1,v_2$ are a $0$-separation lines of $f$, and hence all points of $M_1(f)$ lie between $l$ and $l'$.
Note that for every threshold function $h$ the sets $\textup{Conv}(M_1(h))$ and $\textup{Conv}(M_0(h))$ do not intersect.
These facts imply that $M_1(f) \cap M_1(g)$ should lie between $l$ and $l'$ and between the segments $u_1v_2$ and $u_2v_1$, which proves (\ref{eq:Int_P}). 
Now it follows from (\ref{eq:Int_P}) and the condition of the theorem that $f \notin G_2$ and $\{f_1,f_2\}$ is a unique defining set of $f$.

Finally, we are interested in a point $x' \in E_n^2$ which would distinguish $f$ from every function in $G_2$, that is $f(x') \neq g(x') \text{ for all } g \in G_2$.
By (\ref{eq:Int_P}) we have  $g(x)=0$ for all $g \in G_2$ and $x \in M_1(f) \setminus \textup{Int}(P)$.
Since $B(\textup{Conv}(E_n^2)) \cap \textup{Int}(P) = \emptyset \neq B(\textup{Conv}(E_n^2)) \cap M_1(f)$, we can take an arbitrary point from $B(\textup{Conv}(E_n^2)) \cap M_1(f)$ as $x'$ and obtain a teaching set $T$ for $f$ which is equal to $S(f_1) \cup S(f_2) \cup \{ x' \}$.
Note that any such a teaching set $T$ is minimal and $|T| \leq 9$.

\begin{figure}[p]
    \centering
    \includegraphics[width=0.48\textwidth]{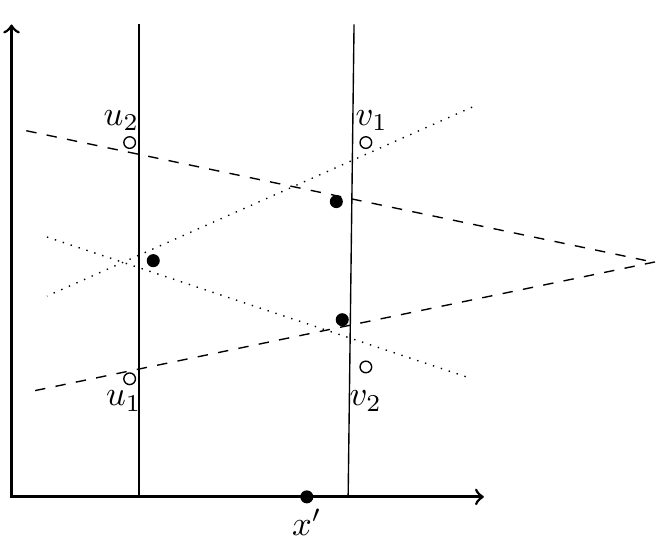}
    \caption{The dashed and dotted lines correspond to pair of different functions from $G_2$, and the solid lines correspond to the function $f$.}
    \label{fig:p9}
\end{figure}

%
%
%
%
%

\end{proof}

\begin{remark}
\label{rem:rem1}
	Theorem \ref{theorem:statem-9} also holds when the domain 
    is a convex subset of $E_n^2$.
\end{remark}

\begin{corollary}
\label{cor:cor-9}
Let $f \in \mathfrak{T}(2, n, 2)$ and there is a unique set of threshold functions 
$\{f_1,f_2 \}$ defining $f$. If $M_1(f) \cap B(\textup{Conv}(E_n^2)) \neq \emptyset$,
then
$$
	\sigma(f, \mathfrak{T}(2, n, 2)) \leq 9.
$$
\end{corollary}
\begin{proof}
By Proposition \ref{prop:statem3}, for $f_1$ and $f_2$ the following is true:
$$
\begin{cases}
	S(f_1) \cap M_0(f_2) = \emptyset; \\
	S(f_2) \cap M_0(f_1) = \emptyset.
\end{cases}
$$
Therefore $f$ satisfies the conditions of Theorem \ref{theorem:statem-9}.
\end{proof}

Recall that $J(f, C)$ denotes the number of minimal teaching
sets of a function $f$ with respect to a class $C$. Using the set of functions $G_2$ from Theorem \ref{theorem:statem-9} the next lemma proves that number of minimal teaching sets of 2-threshold functions can grow as $\Omega(n^2)$.

\begin{lemma}
$$
	\max_{f \in \mathfrak{T}(2, n, 2)}J(f, \mathfrak{T}(2, n, 2)) = \Omega(n^2).
$$
\end{lemma}

\begin{proof}
    Let 
    $$
    		m=m(n) = \left\lfloor \frac{n-1}{4} \right\rfloor.
    $$

	For $n \geq 21$ let $f^{(n)} \in \mathfrak{T} (2,n,2)$ be defined by threshold functions $f_1^{(n)}$ and $f_2^{(n)}$ with the corresponding
	inequalities:
	$$
		\begin{cases}
			-3x_1-4x_2 \leq -25, \\
			3x_1+4x_2 \leq 12m - 1.
		\end{cases}
       $$

Note that $l: 3x_1+4x_2 = 12m - 1$ is a $1$-separation line of $f_2^{(n)}$ and by Proposition \ref{prop:sepline2} we have $\textup{Vert}(\textup{Conv}(l \cap E_n^2)) \subseteq S_1(f_2^{(n)})$.
By construction, $l$ is not parallel to $x_2$-axis and contains at least two integer points from $E_n^2$. 
Therefore $\textup{Conv}(l \cap E_n^2)$ is a segment and its vertices are solutions of the following two integer linear programming problems with constraints $n \in \mathbb{Z}$, $n \geq 21$ and $m = \left\lfloor \frac{n-1}{4} \right\rfloor$:
$$
	\begin{cases}
		\max x_1, \\
		3x_1+4x_2 = 12m - 1, \\
		0 \leq x_1 \leq n-1, \\
		0 \leq x_2 \leq n-1, \\
		x_1,x_2 \in \mathbb{Z},
	\end {cases}
	\begin{cases}
		\min x_1, \\
		3x_1+4x_2 = 12m - 1, \\
		0 \leq x_1 \leq n-1, \\
		0 \leq x_2 \leq n-1, \\
		x_1,x_2 \in \mathbb{Z}.
	\end{cases}
$$

It is easy to check that the solutions of the above problems are $(4m-3,2)$ and $(1,3m-1)$, therefore $S_1(f_2^{(n)}) = \{(4m-3,2),(1,3m-1)\}$.

Now, the closest parallel to $l$ line, which contains $0$-values of $f$, is $l': 3x_1 + 4x_2 = 12m$.
By Proposition \ref{prop:sepline3} all points of $S_0(f_2^{(n)})$ are vertices of $\textup{Conv}(l' \cap E_n^2)$, and to find $S_0(f_2^{(n)})$ we can use the same arguments as we did for $S_1(f_2^{(n)})$.
The same is true for $f_1^{(n)}$ and the set $S(f_1^{(n)})$, hence the final conclusion looks as follows:
	$$
	\begin{cases}
		S_0(f_1^{(n)})=\{u_1=(8,0),u_2=(0,6)\}, \\
		S_1(f_1^{(n)})=\{u_3=(7,1),u_4=(3,4)\}, \\
		S_0(f_2^{(n)})=\{v_1=(0,3m),v_2=(4m,0)\}, \\
		S_1(f_2^{(n)})=\{v_3=(4m-3,2),v_4=(1,3m-1)\}.
	\end {cases}
	$$
	Note that $f^{(n)}$ satisfies the conditions of Corollary \ref{cor:cor-9} and Theorem \ref{theorem:statem-9}
	and quadrilateral $P$ from the proof of Theorem \ref{theorem:statem-9} has vertices $u_1,u_2,v_1$, and $v_2$.
	Denote by $G \subseteq \mathfrak{T}(2, n, 2)$ the set of functions such that for 
	every $g \in G$ and for some threshold functions $g_1, g_2$ defining $g$ 
	the following is true:
$$
	S_1(f_1^{(n)}) \cup S_1(f_2^{(n)}) \subseteq M_1(g),
$$
$$
	\{u_1, v_2\} \subset M_0(g_1), 
$$
$$
    \{u_2, v_1\} \subset M_0(g_2).
$$
The set $G$ corresponds to the set $G_2$ from the proof of Theorem \ref{theorem:statem-9},
therefore all functions from $G$ and only them agree with $f^{(n)}$ on $S(f_1^{(n)}) \cup S(f_2^{(n)})$. 
Let us bound from below the number of points $x'$ such that 
\begin{equation}\label{eq:p9_1}
	f^{(n)}(x') \neq g(x') \text{ for all } g \in G.
\end{equation}
Denote by $R(n)$ the triangle with vertices $v_3$, $v_4$ and $(n-1, n-1)$ and by $L(n)$ the segment $v_3v_4$.
It is clear that $R(n) \cap M_1(f^{(n)}) = L(n) \cap E_n^2$.
By construction of the set $G$, the inclusion  $R(n) \cap E_n^2 \subset M_1(g)$ holds for any $g \in G$.
It means that any point from $(R(n) \setminus L(n)) \cap E_n^2$ distinguishes $f^{(n)}$ from any function in $G$. 
Therefore the number of minimal teaching sets for $f^{(n)}$ can be
lower bounded by the cardinality of $(R(n) \setminus L(n)) \cap E_n^2$, which is equal
to $|R(n) \cap E_n^2| - |L(n) \cap E_n^2|$.

The number of integer points in $L(n)$ can be calculated through the $\textup{GCD}$ of the differences between coordinates of $v_3$ and $v_4$:
$$
	|L(n) \cap E_n^2| = 2 +\textup{GCD}((4m-3)-1, (3m-1)-2) = m+1.
$$
The number of integer points in $R(n)$ can be calculated by means of the Pick's formula.
Indeed, since $R(n)$ is a triangle with integer vertices, we have
$$
\mathcal{S}(R(n)) = |\textup{Int}(R(n))| + \frac{|B(R(n))|}{2} - 1
$$
and therefore
$$
|R(n) \cap E_n^2| = |\textup{Int}(R(n))| + |B(R(n))| 
= \mathcal{S}(R(n)) + \frac{|B(R(n))|}{2} + 1.
$$
Now since
$$
|B(R(n))| \geq |L(n) \cap E_n^2| + 1 \geq m+2
$$
and
$$
\mathcal{S}(R(n)) = \frac{\left|(m - 1)(12m - 7n + 6)\right|}{2},
$$
we conclude that
$$
|(R(n) \setminus L(n)) \cap E_n^2| \geq \frac{\left|(m - 1)(12m - 7n + 6)\right|}{2} + 
\frac{m+2}{2} + 1 - (m+1) = \Theta(n^2).
$$
That is the number of minimal teaching sets for the function $f^{(n)}$ grows as $\Omega(n^2)$.

\end{proof}

\begin{figure}[p]
    \centering
    \includegraphics[width=0.8\textwidth]{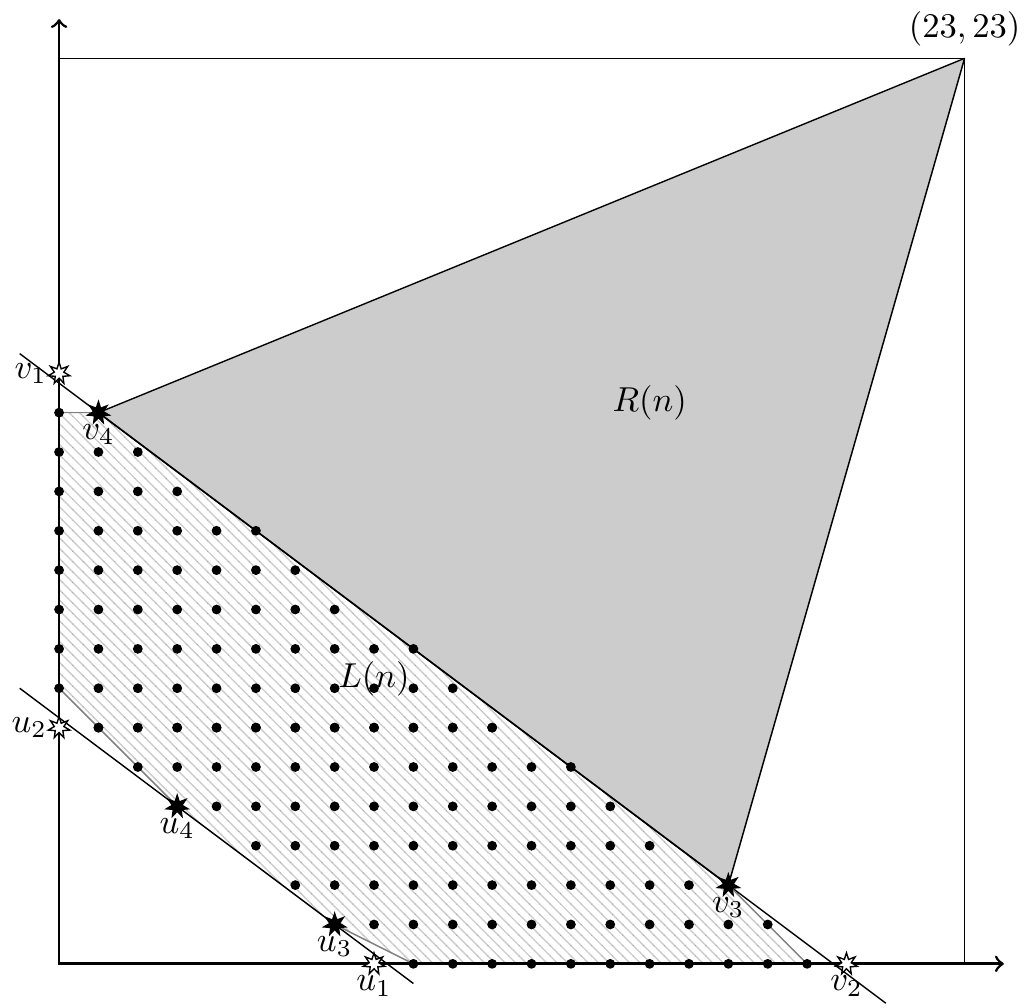}
    \caption{An example of $f^{(24)}$, the black points are the points of $M_1(f)$, the gray shape is $R(n)$, the stripped area is $P(f)$.}
    \label{fig:p9-col}
\end{figure}

\section{Open problems}
In this paper, we investigated structural and quantitative properties of
sets of essential points and minimal teaching sets of $k$-threshold functions.

We proved that a function in the class $\mathfrak{T}(d,n,*)$ has a unique minimal teaching set which is equal to the set of essential points of this function with respect to the class. 
For a function in the class $\mathfrak{T}(2,n,*)$ we estimated the cardinality of the set of essential points of the function. 
It would be interesting to find analogous bounds on the cardinality of the set of essential points of a function in $\mathfrak{T}(d,n,*)$ for $d>2$.

We considered $\mathfrak{T}(2,n,2)$ and proved that the set of essential points of a function in this class is not necessary a minimal teaching set. 
Moreover we showed that $J(\mathfrak{T}(2,n,2))=\Omega(n^2)$. 
Also in the class $\mathfrak{T}(2,n,2)$ we identified functions with minimal teaching sets of cardinality at most $9$. It would be interesting to estimate the proportion of functions with this property in the class $\mathfrak{T}(2,n,2)$.

\section{Acknowledgements}
I thanks the referees for the careful reading of the paper and many helpful sugges-
tions that considerably improved its presentation.


\vskip5ex
\noindent
\textbf{References}

\begin{thebibliography}{12}
%

\bibitem{Angluin}
    D. Angluin,
		Queries and Concept Learning,
		Machine Learning. 1988, V.1, №4, pp 319-342
		
\bibitem{Hegedus}
		T. Heged\"us,
    Generalized teaching dimensions and the query complexity of learning.
		Proc. 8th Ann. ACM Conf. of Computational Learning Theory (COLT'95)
    New York: ACM Press, 1995, pp 108-117
		
\bibitem{Zolotykh}
    N. Yu. Zolotykh, A. Yu. Chirkov,
    On the number of irreducible points in polyhedra, Graphs and Combinatorics, 2016, DOI: 10.1007/s00373-016-1683-1
		
\bibitem{Maass}
		W. J. Bultman, W. Maass,
		Fast Identification of Geometric Objects with Membership Queries,
		Information And Computation 118, 1995, pp 48-64

\bibitem{Brightwell}
		M. Anthony, G. Brightwell, J. Shawe-Taylor, 
		On specifying boolean functions by labelled examples,
		Discrete Applied Mathematics. 1995,  V.61, I.1, pp 1-25

\bibitem{Shevchenko}
    V. N. Shevchenko, N. Yu. Zolotykh,
		Lower Bounds for the Complexity of Learning Half-Spaces with Membership Queries,
		Lecture Notes in Computer Science. 1998, V.1501, pp 61-71

\bibitem{Shevchenko2}
		V. N. Shevchenko, N. Yu. Zolotykh,
    On the complexity of deciphering the threshold
		functions of k-valued logic, Dokl. Math. 1998, V.58, pp 268-270
		
\bibitem{Shevchenko3}
		N. Yu. Zolotykh, V. N. Shevchenko,
    Estimating the complexity of deciphering a
		threshold functions in a k-valued logic, Comput. Math. Math. Phys. 1999, V.39, pp 328-334
    
\bibitem{Euler}
		G. H. Hardy, E. M. Wright,
		An Introduction to the Theory of Numbers — fourth edition.
		Oxford: Oxford University Press, 1975, p 268
		
\bibitem{Trainin}
		J. Trainin,
    An elementary proof of Pick's theorem,
    Mathematical Gazette. 2007, V.91 (522), pp 536-540

\bibitem{Zolotykh2}
		M. A. Alekseyev, M. G. Basova, N. Yu. Zolotykh,
    On the minimal teaching sets of two-dimensional threshold functions,
		SIAM J. Discrete Math. 2015, V.29 (1), pp 157-165
		
		
\end{thebibliography}
\end{document}